\documentclass[draft]{amsart}
\usepackage[utf8]{inputenc}
\usepackage{amssymb}
\usepackage{mathrsfs}
\usepackage{hyperref}
\usepackage{float}
\usepackage{esdiff}   
\usepackage{dsfont} % for bold 1
\usepackage{tikz-cd}

\newtheorem{theorem}{Theorem}[section]
\newtheorem{corollary}[theorem]{Corollary}
\newtheorem{lemma}[theorem]{Lemma}
\newtheorem{proposition}[theorem]{Proposition}

\theoremstyle{definition}
\newtheorem{definition}[theorem]{Definition}
\theoremstyle{remark}
\newtheorem{remark}[theorem]{Remark}
\newtheorem*{example}{Example}

\newcommand{\norm}[1]{\left\|#1\right\|}

\renewcommand{\epsilon}{\varepsilon} 
\newcommand{\Rr}{\mathbb{R}}

\newcommand{\Nn}{\mathbb{N}}

\newcommand{\Zz}{\mathbb{Z}}

\newcommand{\Ff}{\mathbb{F}}
\newcommand{\Kk}{\mathbb K}

\newcommand{\set}[2]{\{\,#1 \, : \, #2\,\} }

\newcommand{\Ob}{\operatorname{Ob}}
\newcommand{\Hom}{\operatorname{Hom}}

\newcommand{\id}{\operatorname{id}}
\newcommand{\cat}[1]{\mathbf{#1}}

\newcommand{\e}{\mathrm{e}}

\newcommand{\U}{ 1}

\newcommand{\sign}{\operatorname{sign}}

\title[A combinatorial approach to categorical  Möbius inversion]{A combinatorial approach to categorical  Möbius inversion and pseudoinversion}
\author{Juan Pablo Vigneaux}
\date{\today}

\begin{document}

\maketitle

\begin{abstract}
    We use Cramer's formula for the inverse of a matrix and a combinatorial expression for the determinant in terms of paths of an associated digraph (which can be traced back to Coates) to give a combinatorial interpretation of  M\"obius inversion \emph{whenever it exists}. Every M\"obius coefficient is a quotient of two sums, each indexed by certain collections of paths in the digraph. Our result contains, as particular cases, previous theorems by Hall  (for posets) and Leinster (for skeletal categories whose idempotents are identities). A byproduct is a novel expression for the magnitude of a metric space as  sum over self-avoiding paths with finitely many terms.  By means of Berg's formula, our main constructions can be extended to Moore-Penrose pseudoinverses, yielding an analogous combinatorial interpretation of M\"obius pseudoinversion and, consequently, of the magnitude of an arbitrary finite category.
\end{abstract}

\section{Introduction}

Magnitude \cite{Leinster2008,Leinster2013} is an invariant of enriched categories that was introduced by  Leinster and has attracted considerable attention in recent years (for a full bibliography, see \cite{magnitude-bibliography}). In a precise sense, it extends the notion of Euler characteristic beyond set-theoretic topology.  When applied to metric spaces, seen as categories enriched over the poset $([0,\infty],\geq)$,  magnitude provides a novel isometric invariant that encodes geometric information such as volume, surface area, and fractal dimension \cite{Meckes2015,Barcelo2018,Gimperlein2021}. 

Given a finite category $\cat A$, let $\zeta:\Ob \cat A \times \Ob \cat A\to \Rr$ be the matrix with coefficients $\zeta(a,b) = |\Hom(a,b)|$. 
The magnitude of $\cat A$ can be defined as the sum of the components of the Moore-Penrose pseudoinverse $\zeta^{+}$ \cite{Akkaya2023,Chen2023formula}.  When $\cat A$ is an ordinary category, each $\Hom(a,b)$ is a set and $|\cdot|$ denotes cardinality. When $\cat A$ is enriched, $\Hom(a,b)$ belongs to a  monoidal category and $|\cdot|$ is an appropriate measure of size. 

The particular case of posets, which can be regarded as ordinary categories whose $\Hom$-sets have at most one element, was treated in  Rota's foundational work on M\"obius inversion \cite{Rota1964}. In this case, the matrix $\zeta$ is invertible and the components of $\zeta^+=\zeta^{-1}$ are known as M\"obius coefficients. The significance of the M\"obius coefficients and hence the magnitude in Rota's combinatorial theory derived from the following theorem (that he attributed to Philip Hall): when $\cat A$ is a locally finite poset, 
\begin{equation}\label{eq:hall-formula}
   (\zeta^{-1})(a,b) = \sum_{k\geq 0} (-1)^k \#\{\text{``paths'' $a=c_0 <\cdots <c_k=b $}\}. 
\end{equation}
Rota's proof follows from the formal expansion
\begin{equation}\label{eq:infinite_expansion}
    \zeta^{-1} = \sum_{k\geq 0} (-1)^k (\zeta-\delta)^k,
\end{equation}
where $\delta$ is the identity matrix (whose components are given by Kronecker's delta). Remark that $\zeta-\delta$ is the adjacency matrix of a directed graph with vertex $\Ob \cat A$ and edges $E=\set{(i,j)}{i<j}$, hence the power $(\zeta-\delta)^k$ counts paths of length $k$ in this graph.  Because the graph is finite and has no directed cycles, the sums in \eqref{eq:hall-formula} and  \eqref{eq:infinite_expansion} have a finite number of terms.

For general categories, however, the sum in \eqref{eq:infinite_expansion} may have infinitely many terms. When the resulting series is absolutely convergent, \eqref{eq:infinite_expansion} gives an explicit formula for the M\"obius coefficients and provides the basis to identify the magnitude as an appropriate alternating sum of Betti numbers in \emph{magnitude homology} \cite[Thm. 7.14]{Leinster2021}.

 Leinster also obtained, by inspection, an alternative generalization of Hall's formula \eqref{eq:hall-formula} that holds for finite skeletal categories that do not accept any cycle $a\to c_1 \to c_2 \to \cdots \to a$ besides the identities $a\to a$ \cite[Thm. 4.1]{Leinster2008}. According to this result, each coefficient  $(\zeta^{-1})(a,b)$ can be expressed as a sum of finitely many terms, see  \eqref{eq:leinsters-formula} below.

In this article, we propose a combinatorial interpretation for the components of $\zeta^{-1}$ that holds \emph{ whenever  $\zeta$ is invertible} and that \emph{in every case only involves sums of finitely many terms}. Instead of using the expansion \eqref{eq:infinite_expansion}, our starting point is Cramer's formula. Given a total ordering $a_1,...,a_n$ of the elements of $\Ob \cat A$, and hence of $\zeta$'s rows and columns, Cramer's formula says that
\begin{equation}\label{eq:Cramer}
    (\zeta^{-1})(a_i,a_j) = \frac{(-1)^{i+j} \det \zeta[j',i']}{\det \zeta},
\end{equation}
where $\zeta[j',i']$ is the submatrix of $\zeta$ obtained by removing the $j$-th row  and the $i$-th column.
By expressing each determinant as a sum over permutations and reinterpreting each term in graph-theoretic terms along the lines of \cite{Brualdi2008}, we can rewrite  \eqref{eq:Cramer} as:
\begin{equation}\label{eq:Cramer_combinatorial}
    (\zeta^{-1})(a_i,a_j) = \frac{\sum_{C\in \mathcal C(i\to j;D(\cat A))} \sign(C)w(C)}{\sum_{L\in \mathcal L(D(\cat A))} \sign(L)w(L)}.
\end{equation} 
In this formula, the denominator's sum is indexed by finitely many  \emph{linear subdigraphs} of a digraph $D(\cat A)$ associated with $\cat A$, and the numerator's sum by finitely many \emph{connections} of $i$ to $j$ in $D(\cat A)$.
The general definition of linear subdigraphs and connections, along with their weights ($w$) and their signatures ($\sign$), can be found in Sections \ref{sec:digraphs}. The digraph  $D(\cat A)$ is described in Section  \ref{sec:comb_int_cats}. The identity \eqref{eq:Cramer_combinatorial} is the subject of Theorem \ref{thm:combinaotiral-Moebius}. The resulting expression for $(\zeta^{-1})(a_i,a_j) $ \emph{does not depend on the total order chosen above.}

Given a metric space $(X,d)$, one can introduce a $([0,\infty],\geq)$-enriched category $\cat A_{X,d}$ such that $\Ob \cat A_{X,d}=X$ and $\Hom(x,y) = d(x,y)$. In this case, the digraph $D(\cat A)$ is the complete digraph with vertex set $X$, and $\zeta(x,y) = \exp(-d(x,y))$.  
Proposition \ref{prop:det_and_moebius_metric_spaces} specializes our quotient formula to this case. We shall see that many of the terms in the involved sums cancel with each other in a way that is reminiscent of boundaries in magnitude homology (where colinearity plays a key role), which suggests the possibility of expressing each coefficient of $\zeta^{-1}$ as a quotient of alternating sums of Betti numbers. 

Our results also entail a novel formula for the magnitude of a finite metric space $(X,d)$ that expresses it as a sum over ``self-avoiding'' paths with finitely many terms. 

\begin{proposition} Let $(X,d)$ be a finite metric space. Then,
\begin{equation}\label{eq:formula_magnitude_paths}
    \chi(\cat A_{X,d})    = \sum_{k=0}^{\# X-1} \sum_{X'=\{x_0,x_1,...,x_k\}\subset X} (-1)^{k} \frac{\det \zeta_{X\setminus X'} }{\det \zeta_{X}}\sum_{\substack{k\text{-paths }\gamma\\
    \text{with vertices }X'}} w(\gamma), 
\end{equation}
where the weight of a path $\gamma:x_{i_1}\to x_{i_2} \to \cdots \to x_{i_k}$ is
\begin{equation}
    w(\gamma) = e^{- d(x_{i_1},x_{i_2})} \cdots e^{- d(x_{i_{k-1}},x_{i_k})}.
\end{equation}
By convention, a path of length $0$ is just a single point $\{x_0\}$ with weight $1$, and $\det \zeta_\emptyset = 1$.
\end{proposition}

Section \ref{sec:linear_algebra} below introduces the combinatorial interpretations of the matrix determinant and the matrix inverse. Section \ref{sec:Mobius-inversion} specializes these results to M\"obius inversion. In Section 3.1 we treat the case of ordinary categories and recover previous results by Hall-Rota and Leinster. In Section 3.2 and 3.3. we cover enriched categories and metric spaces. Although  this introduces some repetitions, it might benefit readers unfamiliar with enriched categories.  Section \ref{sec:final_rmks} closes the article with some perspectives concerning (i) the generalization of our results to the M\"obius pseudoinverse $\zeta^+$ by means of Berg's formula and (ii) the possible homological simplifications among linear subdigraphs and connections.

\section{Linear algebra}\label{sec:linear_algebra}

\subsection{Basic formulae}\label{sec:linear_algebra-defs}

Let $R$ be a commutative ring and $M=R^{I}$ the module freely generated by a finite index set $I$; the set $I$ is not necessarily ordered. We denote by $(e_i)_{i\in I}$ its canonical basis.

A \emph{matrix of type $(I,I)$ with coefficients in $R$} is a map $\xi:I\times I \to R$; we set $ \xi_{i,j}:=  \xi(i,j) $. The matrix $\xi$ defines an endomorphism $u_\xi:M\to M$ via the formula
\begin{equation}\label{eq:endomorphism_map}
    \forall j\in I, \quad u_\xi(e_j) = \sum_{i\in I} \xi_{i,j} e_i.
\end{equation} 

The module $\Lambda^n(M)$ of alternating $n$-forms is isomorphic to $R$, hence   the map $\Lambda^n(u):\Lambda^n(M)\to \Lambda^n(M),\, x_1\wedge \cdots \wedge x_n \mapsto u(x_1)\wedge \cdots \wedge u(x_n)$ induced by a linear map $u:M\to M$ is an homothety $m\mapsto t m$ for some $t\in R$; this $t$ is, by definition,  the \emph{determinant} of $u$.  The determinant of the matrix  $\xi$ is the determinant of the associated endomorphism $u_\xi$.

\begin{proposition}\label{pop:form-det-permutations} Let $I$ be a finite set,  $\mathfrak S_I$ the set of its permutations, and $\xi$ a matrix of type $(I,I)$. Then 
    $$\det \xi = \sum_{\sigma \in \mathfrak S_I} \sign(\sigma) \prod_{i\in I}  \xi_{\sigma(i),i} = \sum_{\tau \in \mathfrak S_I} \sign(\tau) \prod_{i\in I}  \xi_{i,\tau(i)} .$$
\end{proposition}
\begin{proof}
    The leftmost identity follows from a direct computation of $\Lambda^n(u_\xi)(e_1\wedge \cdots \wedge e_n)$ using \eqref{eq:endomorphism_map}. To prove the rightmost identity, remark that $\prod_{i\in I} \xi_{\sigma(i),i} = \prod_{j\in J} \xi_{j,\sigma^{-1}(j)}$ by a relabelling of the factors, and that $\sign(\sigma) = \sign(\sigma^{-1})$  cf. the proof of Prop. 8 in \cite[Ch. 3, \S 8]{Bourbaki2007algebre}.
\end{proof}

In particular, it holds that the matrix $\xi$ and its transpose have the same determinant:  $\det \xi = \det \xi^T$.

Suppose now that $I$ is totally ordered. We can then represent $\xi$ as a squared array, with rows and columns enumerated by the elements of $I$. 
We denote by $\xi[i',j']$ the submatrix of $\xi$ obtained by removing row $i$ and column $j$. A formula attributed to Laplace says that
\begin{equation}
    \det \xi = \sum_{j=1}^n (-1)^{i+j} \xi_{i,j} \det \xi[i',j'].
\end{equation}
see equation (23) in \cite[Ch.III, \S 8]{Bourbaki2007algebre}.
By definition,  the \emph{cofactor} $\kappa_{i,j}$ is the coefficient in front of $\xi_{i,j}$ in the expansion $\det \xi =\sum_{j=1}^n \kappa_{i,j}\xi_{i,j}$:
\begin{equation}
    \kappa_{i,j} = \diffp{}{{\xi_{i,j}}}\det \xi = (-1)^{i+j}\det \xi[i',j'].
\end{equation}

\begin{proposition}[Cramer's formula]\label{prop:cramer} Let $I$ be a finite totally ordered set and $\xi$ a matrix of type $(I,I)$. If $\xi$ is invertible, then 
$$\xi^{-1}(i,j) = \frac{\kappa_{j,i}}{\det \xi} = \frac{(-1)^{i+j}}{\det \xi} \det \xi[j',i'].$$
\end{proposition}
\begin{proof}
    It also follows from Laplace's development that the \emph{cofactor matrix} $\kappa=(\kappa_{i,j})$ satisfies $\kappa^T \xi = (\det \xi) \delta$, where $\delta:I\times I\to R$ is the identity matrix  \cite[Ch. III, \S 8, eq. (26)]{Bourbaki2007algebre}. Moreover, $\xi$ is invertible if and only if $\det \xi$ is a unit of $R$ \cite[Ch. III, \S 8, Prop. 5]{Bourbaki2007algebre}.
\end{proof}

From Proposition \ref{pop:form-det-permutations} and \ref{prop:cramer}, one can obtain an expression for each coefficient $\xi^{-1}(i,j)$ as a sum over paths in an associated directed weighted graph, which is the subject of Section \ref{sec:matrix_comb_formulae}.  We summarize first the definitions concerning directed graphs that will be used in the rest of the article.

\subsection{Directed graphs}\label{sec:digraphs} We follow  \cite{Brualdi2008} with some slight variations.

A \emph{directed graph (digraph)} consists of a finite set $V$ of \emph{vertices} and a set $E\subset V\times V$ of \emph{edges}. 

An ordered pair $(v,w)\in E$ can be represented as an arrow $v\to w$. We define maps $s:E\to V$ and $t:E\to V$, called respectively source and target, by $s(v,w) = v$ and $t(v,w)=w$. When $t(e)=s(e)$, the edge $e$ is called a \emph{loop}.

Any subset $E'$ of $E$ determines a \emph{spanning subdigraph} $D=(V,E')$.

The subdigraph $D'=(V',E')$ \emph{induced} by a subset $V'$ of $V$ is such that $E'$ contains all the edges $e\in E$ that join vertices in $V'$, that is $E'=E\cap(V'\times V')$.

A \emph{weighted digraph} consists of a digraph $(V,E)$ and numbers $w(e)$ for every edge $e\in E$ called \emph{weights}. 

We can similarly introduce a \emph{multidigraph} if we allow $E$ to be a multiset of elements in $V\times V$, i.e. if we allow multiple edges oriented in the same direction between two vertices.

Given a digraph $D=(V,E)$ and $v\in V$, we define the out-degree $d_D^+(v)$ and the in-degree $d_D^-(v)$ of $v$ in $D$ via the formulas
\begin{equation}
    d_D^+(v) = \set{e\in E}{s(e)=v},\quad
    d_D^{-}(v) =\set{e\in E}{t(e)=v}.
\end{equation}
Remark that a loop $(v,v)\in E$ contributes to both. We omit the subindex $D$ if its clear from the context.

Given $D=(V,E)$ and $u,v\in V$, a \emph{walk} from $u$ to $v$ is a sequence of edges $e_1,...,e_k$, for some $k\in \Nn$, such that $s(e_1)=u$, $t(e_k)=v$ and $t(e_i)=s(e_{i+1})$ for $i=1,...,k-1$. One says that the walk has \emph{length} $k$ and  \emph{visits} vertices $x_i=t(e_i)=s(e_{i+1})$, for $i=1,...,k-1$.  The walk is a \emph{cycle} if $u=v$ and $u,x_1,...,x_k$ are all distinct; it is a \emph{path} if $u, x_1,...,x_k, v$ are all distinct.  We also talk about $k$-walk, $k$-cycle, and $k$-path if we want to emphasize the length of the corresponding walk. 

\begin{definition}[Linear subdigraph]
    A \emph{linear subdigraph} of $D=(V,E)$ is a spanning subdigraph such that each vertex has in-degree 1 and out-degree 1. Therefore, it is a spanning collection of vertex-disjoint cycles. We denote by $\mathcal L(D)$ the set of all linear subdigraphs of $D$. 
\end{definition}

Given a  linear subdigraph $L=(V,E')$ we define its  signature $\sign(L)$ as $(-1)^{\# V+c(L)}$, where $c(L)$ denotes the number of cycles in it, and its weight $w(C)$ as the product of the weights of all the edges that it comprises. 

\begin{definition}[Connection]
    A \emph{connection} of a vertex $v$ to a vertex $w$ in $D$ is a spanning subdigraph $C$ in $D$ with the following properties:\footnote{This is called a $1$-connection in \cite{Brualdi2008}.}
\begin{enumerate}
    \item For all $k\in V\setminus \{v,w\}$, $d_C^+(k)=d_C^-(k)=1$,
    \item $d_C^-(v) = 0$ (no incoming edges at $v$) and $d_C^+(w)=0$ (no outgoing edges at $w$).
    \item If $v\neq w$ then $d_C^+(v)=1$ and $d_C^{-}(w)=1$.
\end{enumerate}
We denote by $\mathcal C(v\to w; D)$ the set of all connections of $v$ to $w$ in $D$. 
\end{definition}

Remark that when $v=w$, a connection is determined by a linear subdigraph of the digraph $D_v$ induced by $V\setminus \{v\}$. When $v\neq w$, a connection is determined by a path from $v$ to $w$ that visits vertices $x_1,...,x_s$ together with a linear subdigraph of the digraph induced by $V\setminus \{v,w,x_1,...,x_s\}$. 

Given a connection $C$, we define its  signature $\sign(C)$ as $(-1)^{\# V+c(C)+1}$, where $c(C)$ denotes the number of cycles in the connection,   and its weight $w(C)$ as the product of the weights of all its edges.

\subsection{Combinatorial interpretations of the determinant and the inverse}\label{sec:matrix_comb_formulae}

   Let $\xi$ be a matrix of type $(I,I)$.
In order to state the aforementioned combinatorial interpretation of the determinant,  we introduce a weighted digraph $D(\xi)$ with vertices  $I$ that has a directed edge $i\to j$ with weight $\xi_{i,j}$ whenever $\xi_{i,j}\neq 0$. 

The results presented in this section can be traced back to Coates and have been thoroughly developed by Brualdi and Cvetkovic in \cite{Brualdi2008}. This reference \emph{defines} the determinant via the combinatorial formula in Proposition \ref{prop:formula_det}. For the convenience of the readers that might be already familiar with the standard definition in Section \ref{sec:linear_algebra-defs}, we have decided to give here self-contained proofs of Propositions \ref{prop:formula_det} and \ref{prop:comb-inverse}.

\begin{proposition}\label{prop:formula_det} Let $I$ be a finite set. For any matrix $\xi$ of type $(I,I)$, 
    \begin{equation}
        \det \xi = \sum_{L\in \mathcal L(D(\xi))} (-1)^{\# I+c(L)} w(L).
    \end{equation}
\end{proposition}
\begin{proof}

    Let $n=\# I$ be the order of the matrix.
    
    Let $L=(I,E')$ be a linear subdigraph of $D(\xi)=(I,E)$.
    Because each vertex has out-degree one, there is a well-defined function $f:I\to I$ such that $(v,f(v))\in E$, and because each vertex has in-degree one, the correspondence $i\mapsto f(i)$ must be a permutation of $I$. The weight of $L$ is $w(L) = \prod_{i\in I} \xi_{i,f(i)}$. Conversely, any \emph{nonvanishing} product of the form $\prod_{i\in I} \xi_{i,\sigma(i)}$, where $\sigma$ a permutation of $I$, is the weight of a unique linear subdigraph $L_\sigma=(V, E_\sigma')$ of $D(\xi)$ obtained by including in $E_\sigma'$ the edge $e=(i,\sigma(i))$ for every $i\in I$; the edge $e$ is part of $D(\xi)$ because $\xi_{i,\sigma(i)}\neq 0$. Hence:
    \begin{equation}
        \det \xi = \sum_{\sigma \in \mathfrak S_I } \operatorname{sign}(\sigma) \prod_{i\in I} \xi_{i,\sigma(i)} = \sum_{L\in \mathcal L(D(\xi)) } \operatorname{sign}(\sigma_L) w(L).
    \end{equation}

The proposition follows by remarking that the cycle decomposition of $\sigma_L$ is represented precisely by the cycles in $L$ (including the loops), that the signature $\operatorname{sign}(\sigma_L)$ of the permutation $\sigma_L$ is the product of the cycles's signatures, and that each cycle of length $s$ has signature $(-1)^{s-1}$ (see \cite[Ch. I, \S 6 nº7]{Bourbaki2007algebre}), hence $
        \operatorname{sign}(\sigma_L)= (-1)^{n-c(L)} =(-1)^{n+c(L)}.$
\end{proof}

\begin{proposition}\label{prop:comb-inverse} Let $I$ be a finite set, and let  $\xi$ a matrix of type $(I,I)$. If  $\xi$ is invertible, then
\begin{equation}\label{eq:formula_inverse_matrix_connections}
     \xi^{-1}(i,j) = \frac{1}{\det \xi} \sum_{C\in \mathcal C(i\to j; \xi)} (-1)^{\# I +c(C)+1} w(C).
\end{equation}
\end{proposition}
\begin{proof}
  We adapt an argument in \cite[pp. 104-105]{Brualdi2008}.  
   
   First, fix $j\in I$.  For each  linear subdigraph $L=(V,E_L)\in \mathcal L(D(\xi))$, there is a unique $i\in I$  such that $e=(j,i)\in E_L$, because $j$ must have out-degree 1. Moreover, $L$ determines a connection $C_L$ of $i$ to $j$, obtained by removing the edge $e$ from $E_L$. Remark that $c(C_L)=c(L)-1$ and $w(L)=\xi_{j,i} w(C_L)$.
   
   Denote by $\mathcal L_{i,j}$ the set of linear subdigraphs $L=(V,E_L)$ of $D(\xi)$ such that $(j,i)\in E_L$. Splitting the sum in Proposition \ref{prop:formula_det}, we conclude that
   \begin{equation}
       \det \xi = \sum_{i\in V} \sum_{L\in \mathcal L_{i,j}}  (-1)^{n+c(L)} w(L) 
       = \sum_{i\in V} \xi_{j,i} \left( \sum_{L\in \mathcal L_{i,j}}  (-1)^{n+c(C_L)+1} w(C_L) \right),
   \end{equation}
hence
\begin{equation}\label{eq:combinatorial_formula_cofactor}
    \kappa_{j,i} = \sum_{L\in \mathcal L_{i,j}}  (-1)^{n+c(C_L)+1} w(C_L).
\end{equation}
by comparison with Lagrange's expansion. 

Since there is a one-to-one correspondence between $C(i\to j;D(\xi))$ and $\mathcal L_{i,j}$, equation \eqref{eq:formula_inverse_matrix_connections} follows form Cramer's formula, Proposition \ref{prop:cramer}. Finally, remark that \eqref{eq:formula_inverse_matrix_connections}  does not depend on the chosen enumeration of $I$.
\end{proof}

\begin{corollary}
   $$ \xi^{-1}(i,j) =\frac{\sum_{C\in \mathcal C(i\to j;D(\xi))} (-1)^{c(C)+1} w(C)}{\sum_{L\in \mathcal L(D(\xi))} (-1)^{c(L)} w(L).}$$
\end{corollary}
It follows from combining Propositions \ref{prop:formula_det} and \ref{prop:comb-inverse}, and it corresponds to \cite[Thm. 5.3.2]{Brualdi2008}.

\section{M\"obius inversion}\label{sec:Mobius-inversion}

\subsection{M\"obius function of ordinary categories} 

In what follows, $\cat{A}$ is assumed to be a finite category with $|\Ob(\cat{A})| = n$ unless otherwise stated. The definitions follow \cite{Leinster2008}.

\begin{definition}
Let $\cat{A}$ be a finite category. Its (coarse) \emph{incidence algebra} $R(\cat A)$ is the set of matrices  $ f:  \text{Ob}(\cat{A}) \times \text{Ob}(\cat{A}) \to \mathbb{Q}$, which is a rational vector space under pointwise addition and scalar multiplication of functions, equipped with the  \emph{convolution product}:
\begin{equation}
\label{eq:convolution} 
   \forall \theta, \phi \in R(\cat{A}), \, \forall a,c \in\Ob(\cat{A}),\quad   (\theta \phi)(a,c) = \sum_{b \in\Ob(\cat{A})} \theta(a,b)\phi(b,c).
\end{equation}
 The identity element for this product is  Kronecker's delta function $\delta$; it is given by the formula
\begin{equation}\label{eq:kron_delta}
    \forall a,b\in \text{Ob}(\cat A), \quad \delta(a,b) = \begin{cases}
        1& \text{ if } a=b\\
        0& \text{otherwise}
    \end{cases}.
\end{equation}
\end{definition}

 Upon enumeration of the elements of $\Ob \cat A$, the elements of $R(\cat A)$ can be written as squared arrays with rational coefficients; under this identification, the convolution product is just the usual product of matrices.

\begin{definition}
The \emph{zeta function} $\zeta$ is an element of $ R(\cat{A})$ given  by $\zeta(a,b) = |\text{Hom}(a,b)|$ for all $a,b \in\Ob(\cat{A})$. If $\zeta$ is invertible (with two-sided inverse) in $R(\cat{A})$, we say that $\cat{A}$ \emph{has M\"{o}bius inversion} and denote its inverse, the \emph{M\"obius function}, by $\mu = \zeta^{-1}$. 
\end{definition}

When a category $\cat A$ has M\"obius inversion, its \emph{Euler characteristic} or \emph{magnitude} $\chi(\cat A)$ is defined as  
\begin{equation}\label{eq:magnitude}
    \chi(\cat A) = \sum_{a,b\in \Ob\cat A} \mu(a,b).
\end{equation}

The topological relevance of the M\"obius function and the magnitude will become clear in the next section.

\subsection{Combinatorial interpretation}\label{sec:comb_int_cats}

Given a finite category $\cat A$ with $n$ objects, we denote by $D(\cat A)$ the digraph associated with its $\zeta$ function. Therefore, the vertices of $D(\cat A)$ correspond to the objects of $\cat A$, and there is an arrow from $a\to b$ with weight $|\Hom(a,b)|$ whenever $\Hom(a,b)\neq \emptyset$. 

\begin{theorem}\label{thm:combinaotiral-Moebius}
    Let $\cat A$ be a finite category with $n$ objects. One has
    \begin{equation}
        \det \zeta = \sum_{L\in \mathcal L(D(\cat A))} (-1)^{n+c(L)} w(L). 
    \end{equation}
   Moreover,  
    if $\cat A$ has M\"obius inversion, then
    \begin{equation}
        \mu(a,b) = \frac{1}{\det \zeta} \sum_{C\in \mathcal C( a\to b; D(\cat A)) } (-1)^{n+c(C)+1} w(C).
    \end{equation}
\end{theorem}
Remember that $w(C)$ (resp. $w(L)$) is the product of the weights corresponding to all the edges that appear in the connection $C$ (resp. linear subdigraph $L$). 
\begin{proof}
  The category $\cat A$  has M\"obius inversion if and only if the matrix $\zeta$  is invertible, so we apply Proposition \ref{prop:comb-inverse}. 
\end{proof}

\begin{example}
    Consider the category represented by 
    \begin{equation}
        \begin{tikzcd}[row sep=huge, column sep=huge]
            a \arrow[loop left, "\id_a"] \ar[r,"f" , shift left=1ex] \ar[r,"g", swap] \ar[dr, "\bar f"] \ar[dr, "\bar g", shift right=1ex, swap] & b \ar[loop right, "\id_b"] \ar[d, "h_1" description, shift right=1ex,swap]\ar[d, "h_2" description, shift left=1.5ex, swap] \ar[d, "h_3" description, shift left=4ex]  \\
            & c \ar[loop right, "\id_c"]
        \end{tikzcd}
    \end{equation}
where $h_i \circ f = \bar f$ and $h_i \circ g = \bar g$ for  $i\in \{1,2,3\}$. \footnote{This could be a subcategory of $\cat{Top}$, where $a$ is a two-point space; $b=c=\Rr^2$; $f$ and $g$ inclusions with images $\{(a_f,0), (-a_f,0)\}$ and $\{a_g,0),(-a_g,0)\}$ respectively, for some $a_f,a_g>0$, $a_f\neq a_g$;  $h_1$ a $\pi$ rotation; $h_2$ a $-\pi$ rotation; and $h_3$ a reflection around the $y$-axis.}

If we order $\Ob \cat A$ as $(a,b,c)$, then we can write
\begin{equation}
    \zeta=\begin{bmatrix}
        1 & 2 & 2 \\
        0 & 1 & 3\\
        0 & 0 & 1 
    \end{bmatrix},
\end{equation}
hence its determinant is 1. To get, for instance, $\zeta^{-1}(2,3)$, we need to determine the connections of $a$ to $c$ in $\cat A$. There are 2 connections with $c(C)=1$ given by associated $1$-walks $\bar f:a\to c$ and $\bar g:a\to c$, and 6 connections with $c(C) = 0$ whose associated $2$-walks are the  compositions of $h_i\circ f$ or $h_i \circ g$, for $i=1,2,3$. Hence $\zeta^{-1}(2,3) = (-1)^{3+1+1} 2 +(-1)^{2+0+1} 6 = 6-2=4$. 

In turn, $\zeta^{-1}(i,i)=1$ for $i=1,2,3$, because the only connection of $i$ to $i$ correspond to the cycle decomposition of $\{1,2,3\}\setminus \{i\}$ given by loops, and the associated sign is $(-1)^{3+2+1}=1$. 
\end{example}

Let $\cat A$ be a category and $a,b\in \Ob \cat A$. An $k$-walk from $a$ to $b$ in $\cat A$ is a diagram
\begin{equation}\label{eq:kwalk-in-cat}
    \begin{tikzcd}
        a = c_0 \ar[r,"f_1"] & c_1 \ar[r, "f_2"] &\cdots  \ar[r, "f_k"] & c_k = b.
    \end{tikzcd}
\end{equation}
This $k$-walk is called a circuit if $a=b$, and it is called \emph{nondegenerate} if no $f_i$ is an isomorphism. We say that the $k$-walk is a $k$-path if the objects $c_0,...,c_k$ are all distinct.

Note that a $k$-walk in $\cat A$ is a $k$-walk  in the underlying multidigraph $M(\cat A)$ that has vertices $\Ob \cat A$ and  edges $\bigcup_{a,b} \bigcup_{f\in \Hom(a,b)} \{ (s(f),t(f))\}$. So the terminology here is compatible with the one introduced above in Section \ref{sec:digraphs} for directed graphs, but is in conflict with Leinster's: he uses $k$-path instead of $k$-walk.

\begin{corollary}\label{cor:vanishing_mu}
    Let $\cat A$ be a finite category with $n$ objects. Suppose $\cat A$ has M\"obius inversion. For any $a,b\in \Ob \cat A$, if $\zeta(a,b)=0$ then $\mu(a,b)=0$.
\end{corollary}
\begin{proof}
    If $\zeta(a,b)=0$ then necessarily $a\neq b$. Also, in this case, the set $C(a\to b;D(\cat A))$ is empty: if there was a walk $a=c_1\to \cdots \to c_k=b$ in $D(\cat A)$, there would be a corresponding walk in $\cat A$ of the form \eqref{eq:kwalk-in-cat}. But the composition $f_k\circ \cdots \circ f_1$ would give a morphism from $a$ to $b$, contradicting that $\zeta(a,b)=|\Hom(a,b)| = 0$.
\end{proof}

\begin{remark}
    One does not really need a category to define the incidence algebra $R(\cat A)$ or a distinguished element such as the zeta function (a weighted digraph would suffice). However, in the preceding proof, the composition of morphisms is crucial. 
\end{remark}

The proof of Corollary \ref{cor:vanishing_mu} is essentially already present in  Leinster's article cf. \cite[Thm. 4.1]{Leinster2008}, except that he works directly with Cramer's formula.   A detailed examination of his argument suggested the general approach that we are proposing  here. 

 Leinster, however, did not use Cramer's formula to deduce Proposition \ref{thm:combinaotiral-Moebius} or a particular case of it. He proved by different means a combinatorial interpretation for the M\"obius inverse of categories that are skeletal and whose only morphisms are idempotents (for instance, a poset). As an illustration of our method, we now show how to recover Leinster's theorem from ours. The following elementary lemma is needed.

\begin{lemma}\label{lem:leinster-idempotents}
    The following conditions on a finite category $\cat A$ are equivalent:
    \begin{enumerate}
        \item Every idempotent in $\cat A$ is an identity.
        \item Every endomorphism in $\cat A$ is an automorphism.
        \item Every circuit in $\cat A$ consists entirely of isomorphisms. 
    \end{enumerate}
\end{lemma}
For a proof, see \cite[Lem. 1.3]{Leinster2008}.

\begin{corollary}[Leinster]\label{cor:leinsters-theorem}
    Let $\cat A$ be a finite skeletal category in which the only idempotents are identities. Then $\cat A$ has M\"obius inversion and
    \begin{equation}\label{eq:leinsters-formula}
        \mu(a,b) = \sum_{\mathrm{paths}\,a=c_0\to \cdots \to c_k=b\,\mathrm{in}\, M(\cat A)} \frac{(-1)^k}{|\Hom(c_0,c_0)| \cdots |\Hom(c_k,c_k)|} 
    \end{equation}
\end{corollary}
Under these assumptions, $\Hom(c,c)$ is the automorphism group of $c\in \Ob \cat A$. Remark also that the paths $c_0 \to \cdots \to c_k$ counted  here are always nondegenerate: since $c_i\neq c_{i+1}$, the morphism $f_i:c_i\to c_{i+1}$ cannot be an isomorphism because $\cat A$ is skeletal.

\begin{proof}
Every cycle in $D(\cat A)$ is a circuit; a circuit consists entirely of isomorphisms according to Lemma \ref{lem:leinster-idempotents} but since $\cat A$ is also assumed to be skeletal, the cycle must be a loop. Hence:
\begin{enumerate}
    \item\label{only-lin-subgraph} The only linear subdigraph of $D(\cat A)$ is the subdigraph $L_0$ given solely by loops. In particular, $c(L_0)=n$.
    \item\label{equivalence-connections} A connection of $a$ to $b$ in $D(\cat A)$  is uniquely determined by a path from $a$ to $b$, because the cycle decomposition of the digraph induced by the unvisited vertices is trivially given by loops.
\end{enumerate}

It follows from the first of these observations that
\begin{equation}
    \det \zeta = (-1)^{n+c(L_0)} w(L_0) = \prod_{a\in \Ob\cat A} |\Hom(a,a)|.
\end{equation}

The weight of the connection $C_\gamma$ determined by the path $\gamma:a=c_0\to \cdots \to c_k=b$ in $D(\cat A)$ is $w(C_\gamma)=|\Hom(c_0,c_1)|\cdots |\Hom(c_{k-1},c_k)| \prod_{c\in \Ob\cat A\setminus\{c_1,..,c_n\} }|\Hom(c,c)|$ and its signature is $\sign(C_\gamma)=(-1)^{n+(n-(k+1))+1}=(-1)^{k}$.

According to Theorem \ref{thm:combinaotiral-Moebius}, for any $a,b\in \Ob\cat A$,
\begin{align}
    \mu(a,b) &= \frac{1}{\det \zeta} \sum_{\text{paths }\gamma:a=c_0\to \cdots \to c_k=b\text{ in }D(\cat A)} (-1)^k w(C_\gamma) \nonumber \\
    &= \sum_{\text{paths }\gamma:a=c_0\to \cdots \to c_k=b\text{ in }D(\cat A)} (-1)^k \frac{|\Hom(c_0,c_1)|\cdots |\Hom(c_{k-1},c_k)|}{|\Hom(c_0,c_0)|\cdots |\Hom(c_k,c_k)|}.
\end{align}
Since $|\Hom(c_0,c_1)|\cdots |\Hom(c_{k-1},c_k)|$ is precisely the number of different paths $c_0\to \cdots \to c_k$ in the multigraph $M(\cat A)$,  \eqref{eq:leinsters-formula} follows. 
\end{proof}

Let $(P,\leq)$ be a finite poset. We introduce an associated category $\cat P$ such that $\Ob\cat P = P$ and 
\begin{equation}
    \Hom(p,q) = \begin{cases}
        \{\ast\} & \text{if }p\leq q \\
        \emptyset & \text{otherwise}
    \end{cases}.
\end{equation}
The category $\cat P$ is such that the only endomorphisms are the identities, so the following theorem|attributed to Philip Hall (see \cite{Rota1964})|ensues from Corollary \ref{cor:leinsters-theorem}, as \cite{Leinster2008} already remarked. 

\begin{corollary}[Hall's theorem]
Let $\cat P$ be a category associated with a finite poset. Then $\cat P$ has M\"obius inversion and 
$$\mu(a,b)=\sum_{k\geq 0} (-1)^k \#\{\text{$k$-paths from $a$ to $b$}\}.
$$
\end{corollary}
Recall again that the $k$-paths counted here are necessarily nondegenerate. 

In a similar vein, it follows from Leinster's theorem (Corollary \ref{cor:leinsters-theorem})  that if $\cat A$ is a skeletal category whose only endomorphisms are identities, then
\begin{equation}
    \chi(\cat A) = \sum_{a,b\in \Ob \cat A} \mu(a,b)
\end{equation}
is the topological Euler characteristic of the geometric realization $B\cat A$ of its nerve $N\cat A$. The nerve $N\cat A$ is an abstract simplicial complex whose $k$-simplices are $k$-paths. The particular case of posets was already treated by Rota in his foundational work on M\"obius inversion.

\subsection{Extension to enriched categories}

Above, we have set $\zeta(a,b)=|\Hom(a,b)|$, that is the \emph{cardinality} of the \emph{set} $\Hom(a,b)$. For further generality, one can suppose that $\Hom(a,b)$ is an object of a certain category $\cat V$ and that $|\cdot|$ represents some way of measuring the ``size'' of such an object. Leinster proposed this extension in \cite{Leinster2013}, which is the source of this section.

To  make sense of morphism composition, we need this category $\cat V$ to come with a suitable product. This leads to the definition of a monoidal category, which  is an ordinary category  $\cat V$ equipped with a binary operation $\otimes:\cat V \times \cat V\to \cat V$ and a unit $\U\in \Ob \cat V$ that satisfy the axioms of a monoid up to coherent isomorphism. 

\begin{example}
    \begin{enumerate}
        \item The category $\cat V = \cat{Set}$ of sets, equipped with the cartesian product  $\otimes=\times$, and the unit $\U = \{\ast\}$.
        \item If $\Kk$ is a field,   the category $\cat V = \Kk\textrm{-}\cat{Vect}$ of $\Kk$-vector spaces, equipped with the tensor product $\otimes_{\Kk}$ and the unit $\U =\Kk$. 
        \item The poset $\cat V =([0,\infty],\geq)$, where $x\to y$ if and only if $x\geq y$, equipped with $\otimes = +$ and $\U = 0$. 
        \item $\cat V = \cat 2$, the category with objects $f$ (false) and $t$ (true) and a single nonidentity arrow $f\to t$, with product $\otimes$ given by logical conjunction and $\U = t$. 
    \end{enumerate}
\end{example}

\begin{definition}
    A $\cat V$-category $\cat A$ is given by a set of objects $\Ob \cat A$ and, for all $a,b\in \Ob \cat A$, an object $\Hom(a,b)$ of $\cat V$, together with $\cat V$-morphisms $\Hom(a,b)\otimes \Hom(b,c)\to \Hom(a,c)$ (compositions) and $\U\to \Hom(a,a)$ (identities), for any $a,b,c\in \Ob \cat A$, that are subject to the usual categorical axioms that ensure associativity of compositions and neutrality of the identities (the axioms hold  only up to coherent isomorphism in $\cat V$). 
\end{definition}

\begin{example}
    \begin{enumerate}
        \item If $\cat V = \cat{Set}$, then a $\cat V$-category is a small category.
        \item If $\cat V = \Kk\textrm{-}\cat{Vect}$, then a $\cat V$-category is a linear category.
        \item If $\cat V = ([0,\infty],\geq)$, then a $\cat V$-category is a generalized metric space in the sense of Lawvere, given by a set of points $A=\Ob\cat A$ and numbers $d(a,b):=\Hom(a,b)\in[0,\infty]$ such that 
        $$d(a,b) + d(b,c)\geq d(a,c).$$
        \item If $\cat V = \cat 2$, then a $\cat V$-category is a set equipped with a reflexive and transitive relation, that is a preorder. In particular every poset is a $\cat 2$-category.
    \end{enumerate}
\end{example}

\begin{remark}\label{rmk:embeddings_cats}
    There is an embedding of monoidal categories $\iota_1:\cat 2\hookrightarrow \cat{Set}$ such that $f\mapsto \emptyset$ and $t\mapsto \{\ast\}$. Similarly, there is an embedding  $\iota_2:\cat 2 \hookrightarrow ([0,\infty],\geq)$ such that $f\mapsto \infty$ and $t\mapsto 0$. In turn, these embeddings induce inclusions $\cat{Posets}\hookrightarrow \cat{Sets}$ and $\cat{Posets}\hookrightarrow \cat{MetricSpaces}$
\end{remark}

Let $\Kk$ be a semiring (a ring without  inverses). We introduce now a \emph{valuation}: a monoid morphism $|\cdot|:(\Ob \cat V/\cong,\otimes, 1)\to (\Kk,\cdot,1)$ defined on ismomorphisms classes of objects in $\cat V$.

\begin{example}
    \begin{enumerate}
        \item When $\cat V = \cat{FinSet}$, we can take $\Kk\supset \Zz$ and  $|X| = \#X$. 
        \item When $\cat V$ consists of finite dimensional $\Ff$-vector spaces for some field $\Ff$, we can take $\Kk\supset\Zz$ and $|X| = \dim_{\Kk}X$. 
        \item Let  $\cat V =([0,\infty],\geq)$ and $\Kk = \Rr$. Since $f(X):=|X|$ must solve the functional equation $f(x+y)=f(x)f(y)$, it must have the form $\e^{-kX}$ for some $k\in \Rr$. 
        \item When $\cat V =\cat 2$ and $\Kk$ any semiring, the only choice is  $|f|=0$ and $|t|=1$. 
    \end{enumerate}
\end{example}

Suppose now that  $\cat{A}$ is  a $\cat V$-category 
 such that  $n=|\Ob(\cat{A})|<\infty$. Let $|\cdot|:(\Ob \cat V/\cong,\otimes, 1)\to (\Kk,\cdot,1)$ denote a valuation for some semiring $(\Kk,\cdot,1)$. The (coarse) \emph{incidence algebra} of $\cat{A}$ is the set 
$R(\cat{A})$ of matrices  $ f:  \text{Ob}(\cat{A}) \times \text{Ob}(\cat{A}) \to \mathbb K$, which is a $\Kk$-semimodule under pointwise addition and scalar multiplication of functions, equipped with the  \emph{convolution product} \eqref{eq:convolution}, whose identity is again \eqref{eq:kron_delta}.

The \emph{zeta function} $\zeta\in R(\cat A)$ is given by $\zeta(a,b)=|\Hom(a,b)|$. If $\zeta$ is invertible, we say that $\cat A$ \emph{has M\"obius inversion} and the inverse $\mu=\zeta^{-1}$ is called the M\"obius function of $\cat A$. When $\cat A$ has M\"obius inversion, its magnitude is again given by \eqref{eq:magnitude}. 

If $\mathbb K$ is a commutative ring, then the formulas for $\det \zeta $ and the components of $\zeta^{-1}$ in Theorem \ref{thm:combinaotiral-Moebius} also hold in this more general framework. The only change is in the interpretation of the weight $|\Hom(a,b)|$ of an edge $a\to b$ in $D(\cat A)$.

\subsection{Magnitude of a metric space}

A finite metric space  $(X,d)$ can be seen as as a $([0,\infty],\geq)$-category $\cat X$ whose objects set is $X$ and such that $\Hom(a,b) = d(a,b)$. We consider valuations $|\cdot|_t :x\mapsto e^{-tx}$ for any $t>0$, which give rise to functions $\zeta_t\in R(\cat A)$ given by $\zeta_t(a,b) = \e^{-td(a,b)}$. 

The digraph $D(X;t):=D(\zeta_t)$ is the \emph{complete} weighted digraph with vertex set $X$ (i.e. the edge set is $X\times X$) that assigns weight $\e^{-t d(x,y)}$ to the edge $x\to y$. We set $D(X) = D(X;1)$.

\begin{proposition}\label{prop:det_and_moebius_metric_spaces}
    Let $\cat X$ be the category corresponding to a finite metric space, $\zeta_t$ the function given by  $\zeta_t(a,b)=\e^{-td(a,b)}$. Then,
        \begin{equation}\label{eq:zeta_metric_with_det}
            \det \zeta_t = \sum_{L\in \mathcal L(D(X;t))} (-1)^{\# X + c(L)} w_t(L).
        \end{equation}
  Moreover, if $\zeta_t$ has an inverse $\mu_t$, then
  \begin{equation}\label{eq:moebius_metric_with_det}
        \mu_t(x,y) = \frac{1}{\det \zeta_t} \sum_{C\in \mathcal C(x\to y; D(X;t)) } (-1)^{\# X + c(C) +1} w_t(C).
    \end{equation} 
\end{proposition}
Recall that $c(C)$ is the number of cycles in the connection $C=(X', E)$ and $w_t(C)$ is the product of all weights $e^{-td(a,b)}$ for each edge $a\to b$ in $C$. Remark  that $w_t(C) = \exp(-t\ell(C))$, where $\ell(C)$ is the length of $C$: the sums of the weights $d(a,b)$ of all the edges $a\to b$ comprised by $C$. The expressions $c(L)$, $w_t(L)$ and $\ell(L)$ have an analogous meaning. 

\begin{remark}
As we explained in the introduction, in some cases one can make sense of the equality 
\begin{equation}\label{eq:infinite_sum_leinster}
        \mu(x,y) = \sum_{k=0}^\infty \sum_{x=x_0 \neq a_1 \neq \cdots \neq x_k = y}(-1)^k \zeta(x_0,x_1) \cdots \zeta(x_{k-1},x_k),
    \end{equation}
    where the second sum ranges over all sequences $(x_0,...,x_k)\in X^{k+1}$ such that $x_0=x$, $x_k=y$ and any two \emph{consecutive} elements are different. For this, one needs to prove that the series in \eqref{eq:infinite_sum_leinster} converges. This is the case, for instance, if $X$ is \emph{scattered }i.e. if for all $x,y\in X$ such that $x\neq y$ one has $d(x,y)>\log(\# X -1) $, see \cite[Prop. 2.1.3]{Leinster2013}. 
    
    In all cases where the sum converges, one can establish the identity $\mu\ast \zeta=\delta$  by direct computation, hence \eqref{eq:infinite_sum_leinster} is a valid expression for the M\"obius inverse, which in turn entails a similar expression for the magnitude which provides the foundation of magnitude homology \cite{Leinster2021}. At the moment, we do not see an explicit relation between these infinite sums and the finite sums in Proposition \ref{prop:det_and_moebius_metric_spaces}. 
\end{remark}

\begin{definition}
    The \emph{magnitude function} associated with the finite metric space $(X,d)$ is $t\mapsto |tX|:=\sum_{x,y\in X} \mu_t(x,y)$, defined for any $t>0$.
\end{definition}

A simple application of Proposition \ref{prop:det_and_moebius_metric_spaces} is a new proof of a well-known fact: 

\begin{corollary}
    $\lim_{t\to \infty}|tX| =\# X$.
\end{corollary} 
\begin{proof}
    Let $C$  be a connection, with edges  $E\subset X\times X$. Remark that $w_t(C)\to 0$ unless $C\in \mathcal C(x\to x;D(X))$ for some $x\in X$ and $E$ consists only of loops, in which case $w_t(C) = 1$ for all $t$. It follows that $$\lim_{t\to\infty} \sum_{x,y\in X}\sum_{C\in \mathcal C(x\to y; D(X;t)) } (-1)^{\# X +  c(C) +1} w_t(C) = \sum_{x\in X} (-1)^{\# X + (\#X-1)+1} = \# X.$$ Similarly, if $L=(X,E_L)$ is a linear subdigraph of $D(X)$, then $w_t(L)\to 0$ unless $E_L$ consists solely of loops. There is only one such linear subdigraph, and it has signature $1$ and weight $1$.
\end{proof}

 A metric space $(X,d)$ has the \emph{one-point} property if $\lim_{t\to 0} |tX| =1$. The space of all $n$-point metric spaces contains a dense open subset on which the one-point property holds \cite{Roff2023small}; in particular, it holds for every nonempty compact subset of a finite dimensional subspace of $L_1[0,1]$, such as  $\ell_1^n$ or the Euclidean space of dimension $n$ \cite{Leinster2023spaces}. The known limits of $|tX|$ at $0$ and $\infty$ have justified regarding $|tX|$ as the ``effective number of points'' of the rescaled metric space $(X,td)$. However,  under that interpretation, the possible negativity of $|tX|$ (see \cite[Ex. 2.2.7]{Leinster2013}) is puzzling.  Proposition \ref{prop:det_and_moebius_metric_spaces} gives a new meaning to the magnitude's sign: it comes from the signatures of the connections and linear subdigraphs.

\begin{remark}[Some identities involving  lengths]
Let $\kappa_t$ denote the cofactor matrix of $\zeta_t$ (see Section \ref{sec:linear_algebra-defs}). In view of the expansion
\begin{equation}
w_t(G) = \exp(-t\ell(G)) = \sum_{k=0}^{\infty} \frac{\ell(G)^k (-t)^k}{k!},
\end{equation}
and the combinatorial expression for the components of $\kappa_t$ in \eqref{eq:combinatorial_formula_cofactor}, we conclude that
\begin{equation}\label{eq:expansion-cofactormat-lengths}
\sum_{x,y\in X} \kappa_t(x,y) = \sum_{k=0}^\infty \left( \sum_{C\in \mathcal C(D(X))} \frac{\sign(C) \ell(C)^k}{k!} \right) (-t)^k,
\end{equation}
where $\mathcal C(D(X)) = \bigcup_{x,y\in X} \mathcal C(x\to y;D(X))$. Similarly
\begin{equation}\label{eq:expansion-det-lengths}
\det(\zeta_t) = \sum_{k=0}^\infty \left( \sum_{L\in \mathcal L(D(X))} \frac{\sign(L) \ell(L)^k}{k!} \right) (-t)^k.
\end{equation}

Let us order the elements of $X$ and denote by $\mathbf d_i^{(k)}$ the column vector $(d(x_i,x_1)^k,...,d(x_i,x_n)^k)^T$. 
By introducing a series expansion of each component of $\zeta_t$, Roff and Yoshinaga  \cite{Roff2023small} proved that 
\begin{equation}\label{eq:expansion-cofmat-roff}
\sum_{x,y \in X} \kappa_t(x,y) = F_n(d) (-t)^{n-1} + C_n (-t)^{n} + o(t^n)
\end{equation}
and 
\begin{equation}\label{eq:expansion-det-roff}
\det(\zeta_t) = F_n(d) (-t)^{n-1} + C'_n (-t)^{n} + o(t^n),
\end{equation}
where 
\begin{align}
F_n(d) &= (-1)^{n-1} \sum_{j=1}^n \det( \mathbf{d}_1^{(1)},\dots,\mathbf d_j^{(0)},\dots, \mathbf{d}_n^{(1)}), \\
C_n &= \frac{1}{2} \sum_{\substack{i,j=1\\i\neq j}}^n \det( \mathbf{d}_1^{(1)},\dots,\mathbf{d}_i^{(0)},\ldots,\mathbf d_j^{(2)},\dots, \mathbf{d}_n^{(1)}),\\
C_n' &= C_n + \det(d)
\end{align}
By comparing the coefficients in \eqref{eq:expansion-cofactormat-lengths}-\eqref{eq:expansion-det-roff}, we obtain several identities that involve the lengths of connections and linear subdigraphs. Firstly,  for all $k\in \{0,...,n-2\}$, 
\begin{equation}
\sum_{C\in \mathcal C(D(X))} \sign(C) \ell(C)^k = \sum_{L\in \mathcal L(D(X))} \sign(L) \ell(L)^k = 0.
\end{equation}
Moreover,
\begin{align}
\sum_{C\in \mathcal C(D(X))} \sign(C) \ell(C)^{n-1} &= \sum_{L\in \mathcal L(D(X))} \sign(L) \ell(L)^{n-1} = (n-1)! F_n(d),\\
\sum_{C\in \mathcal C(D(X))} \sign(C) \ell(C)^n &= n!C_n,\\
\sum_{L\in \mathcal L(D(X))} \sign(L) \ell(L)^n &= n! C'_n.
\end{align}

\end{remark}

\subsection{A  formula involving paths}

One can write a formula for the magnitude directly in terms of weighted paths, at the cost of using not only the determinant of its function $\zeta_t$ but also of the zeta functions of all its subspaces $(\tilde X, d)$, where $\tilde X\subset X$.  We write $\det \zeta_{\tilde X,t}$ to make the dependence on both $\tilde X$ and $t$ explicit. 

\begin{proposition} Let $(X,d)$ be a finite metric space. Then,
\begin{equation}
    |tX|    = \sum_{k=0}^{\# X-1} \sum_{X'=\{x_0,x_1,...,x_k\}\subset X} (-1)^{k} \frac{\det \zeta_{X\setminus X',t} }{\det \zeta_{X,t}}\sum_{\substack{k\text{-paths }\gamma\\
    \text{with vertices }X'}} w_t(\gamma), 
\end{equation}
where the weight of a path $\gamma:x_{i_1}\to x_{i_2} \to \cdots \to x_{i_k}$ is
\begin{equation}
    w_t(\gamma) = e^{-t d(x_{i_1},x_{i_2})} \cdots e^{-t d(x_{i_{k-1}},x_{i_k})}.
\end{equation}
By convention, a path of length $0$ is just a single point $\{x_0\}$ with weight $1$, and $\det \zeta_\emptyset = 1$.
\begin{proof}
    Multiply \eqref{eq:moebius_metric_with_det} by $\det \zeta_t = \det \zeta_{X,t}$ and then sum over all $x,y\in X$ to get
    \begin{equation}
        |tX| \det \zeta_{X,t} = \sum_{x,y\in X} \sum_{C\in \mathcal C(x\to y; D(X;t)) } (-1)^{\# X + c(C) -1} w(C)
    \end{equation}
    
    Any connection $C$ from $x$ to $y$ consists of a path $\gamma:x=x_0\to x_1 \to \cdots \to x_k=y$, for some $0\leq k\leq \# X -1$, and a linear subdigraph $L$ of the digraph induced by $X\setminus X'$, where $X'=\{x_0,...,x_k\}$. By convention, when $X\setminus X'=\emptyset$, we suppose there is a unique linear subdigraph $L_\emptyset$ with $c(L_\emptyset)=0$ and $w(L_\emptyset)=1$.It holds in general that $w(C) = w(\gamma)w(L)$ and $(-1)^{\# X +c(C)+1} = (-1) ^{{k+1}+(\# X-(k+1))+c(L) +1}$. Therefore one has
    \begin{multline}
        |tX| \det \zeta_{X,t} = \\\sum_{k=0}^{\# X -1} \sum_{\substack{k-\text{path } \gamma: \\ x_0 \to \cdots \to x_k}} (-1)^k w(\gamma) \left( \sum_{L\in \mathcal L(D(X\setminus \{x_0,...,x_k\};t))} (-1)^{(\# X-(k+1))+c(L)} w(L)\right)
    \end{multline}
    The term in parenthesis is precisely $\det \zeta_{X\setminus X';t}$, again in virtue of Proposition \ref{prop:det_and_moebius_metric_spaces}. The result follows from a rearrangement of the terms.
\end{proof}
     
\end{proposition}

\section{Perspectives}\label{sec:final_rmks}

\subsection{Pseudoinversion}

In many cases, the matrix $\zeta$ associated to a category $\cat A$ is not be invertible. In fact, this is always the case if $\cat A$ is not skeletal. Nonetheless, if $\zeta$ has complex coefficients, then there is a \emph{unique} matrix $\zeta^+$ that satisfies the equations
\begin{equation}
    \zeta \zeta^+ \zeta = \zeta, \quad \zeta^+ \zeta \zeta^+=\zeta^+, \quad (\zeta\zeta^+)^* = \zeta \zeta^+, \quad (\zeta^+\zeta) = \zeta^+\zeta
\end{equation}
called the Moore-Penrose pseudoinverse of $\zeta$ \cite{Penrose1955}.  We call $\tilde \mu := \zeta^+$  the \emph{pseudo-M\"obius function}.  When $\zeta$ is invertible, it is clear that $\zeta^+ = \zeta^{-1}$ and $\tilde \mu = \mu$, the usual M\"obius function. 

One can define the magnitude of \emph{any finite category} $\cat A$ as $\chi(\cat A) = \sum_{a,b\in \Ob\cat A} \tilde \mu(a,b)$, see \cite{Akkaya2023, Chen2023formula}; this definition coincides with the original one given by Leinster using weightings and coweightings when that one applies. Akkaya  and  {\"U}nl{\"u} \cite{Akkaya2023} proved that $\chi$ is invariant under equivalence of categories.

At least part of Propositions \ref{prop:formula_det} and \ref{prop:comb-inverse}, which underlie the combinatorial interpretation of the M\"obius coefficients,   extend to  pseudoinverses by means of Berg's formula.  Let $I$  be a finite set and  $\xi$  a matrix of type $(I,I)$ and rank $r$ (perhaps strictly smaller that $|I|$). Berg  \cite[pp. 122ff]{BenIsrael2003} expressed its Moore-Penrose pseudoinverse in terms of ordinary inverses:
\begin{equation}
    \xi^+ = \sum_{(J,K) \in \mathcal N(A)} a_{J,K} (\xi|_{J\times K})^{-1},
\end{equation}
where $\mathcal N(A)$ consists of pairs $J,K$ such that $J\subset I$, $K\subset I$, $|J|=|K| = r$, and the  matrix $\xi|_{J\times K}$ of type $(J,K)$ obtained by restriction is invertible. The resulting inverse $(\xi|_{J\times K})^{-1}$ is interpreted as a matrix of type $(I,I)$ in an obvious way (induced by the inclusion $J\times K \hookrightarrow I\times I$). In turn,  
\begin{equation}
    a_{J,K} = \frac{\det^2 (\xi|_{J\times K} )}{ \sum_{(J,K)\in\mathcal N(A)} \det^2 (\xi|_{J\times K} ) }. 
\end{equation}

Suppose now that $I$ is totally ordered. Write $J=\{j_1,...,j_r\}$ and $K=\{k_1,...,k_r\}$, for increasing sequences $(j_t)_t$ and $(k_s)_s$. Proposition \ref{prop:cramer} then implies that 
\begin{equation}
    (\xi|_{J\times K})^{-1}(j,k) = \begin{cases} 
    0 &\text{if } j\notin J \text{ or }k\notin K \\
        \frac{(-1)^{s+t} \det \xi|_{J\times K}[t',s']}{\det \xi|_{J\times K}} &\text{if }j=j_s\text{ and } k=k_t \text{ for some }(s,t)
    \end{cases}.
\end{equation}
Set $J'=J\setminus \{j_t\}$ and $K'= K\setminus \{j_s\}$. By virtue of  Proposition \ref{pop:form-det-permutations},
\begin{equation}
    \det \xi|_{J\times K}[t',s'] = \det \xi|_{J'\times K'} = \sum_{\sigma\in \mathfrak S_{r-1}} \sign(\sigma) \prod_{t=1}^{r-1} \xi|_{J'\times K'}(\sigma(t),t). 
\end{equation}

The product $\prod_{t=1}^{r-1} \xi|_{J'\times K'}(\sigma(t),t)$ can be regarded as the weight of a \emph{generalized connection} $C$ in $D(\xi)$: a finite collection of pairwise disjoint paths and cycles. Recall that the vertex set of $D(\xi)$ is $I$. The connection $C$ is such that the vertices that belong to $J'$ have out-degree $1$ and those that belong to $K'$ have in-degree $1$. Hence vertices in $J'\cap K'$ have both out-degree and in-degree one, whereas those in $J'\setminus (J'\cap K')$ have only an outgoing edge in $C$ and those in $K'\setminus (J'\cap K')$ only incomming edge in $C$. It is easy to verify that $a:=|J'\setminus (J'\cap K')| = |K'\setminus (J'\cap K')| $. It follows that $C$  consists of $a$ different paths that only visit vertices in $J'\cap K'$, and that the unvisited points in $J'\cap K'$ form pairwise disjoint cycles. 

Similar remarks hold for the development of $\det \xi|_{J\times K}$ in terms of permutations; one obtains in this way \emph{generalized linear subdigraphs}. We conjecture that it is possible to describe the signature of each permutation purely in terms of topological invariants of its corresponding generalized connection or linear subsigraph. 

\subsection{Homological simplifications}

Magnitude homology \cite{Leinster2021} is an $(\mathbb N,\mathbb R)$-bigraded homology theory  $H_{*,*}(X)$ associated to a metric space $X$. For a given $\ell\in \mathbb R$, the abelian groups $H_{*,\ell}(X)$ are, by definition, the homology groups of the chain complex $(MC_{n,\ell}(X),d_n)_{n\geq 0}$, where
\begin{equation}
    MC_{n,\ell}=\mathbb Z\left[\set{\langle x_0,...,x_n \rangle}{d(x_0,x_1)+\cdots+d(x_{n-1},x_n)=\ell }\right]
\end{equation}
and $d_n:MC_{n,\ell}(X)\to MC_{n-1,\ell}$ is given by an alternating sum $d_n = \sum_{i=1}^{n-1} (-1)^i d^i_n$; each map $d_n^i$ omits an intermediate point provided this does not change the total distance:
\begin{equation}
    d_n^i(\langle x_0,...,x_n \rangle = \begin{cases}
        \langle x_0,..., x_{i-1},x_{i+1},...,x_n \rangle &\text{if }d(x_{i-1},x_i) + d(x_i,x_{i+1}) = d(x_{i-1},x_{i+1})\\
        0 &\text{otherwise}
    \end{cases}.
\end{equation}

Many terms in  \eqref{eq:moebius_metric_with_det} cancel with each other, in a way that is reminiscent of boundaries in magnitude homology \cite{Leinster2021}. We illustrate this with an example. 

\begin{example}\label{rmk:simplification}
    Consider the 6 point metric subspace $X=\{x_1,..,x_6\}$ of $(\Rr^2,\norm{\cdot}_2)$ given by $x_i=(i,0)$ and $x_{i+3}=(i,1)$ for $i=1,2,3$. A connection $C_1$ of $x_1$ to $x_3$ is given by $2$-path $x_1 \to x_2\to x_3$ and loops based at $x_4$, $x_5$ and $x_6$; this connection has positive signature $1=(-1)^{6+3-1}$ and weight $e^{-t1}e^{-t1}$. Another connection, $C_2$, of $x_1$ to $x_3$ is given by the $1$-path $x_1\to x_3$ together with loops at the other points; remark that $\sign(C_2)=-\sign(C_1)$ but $w(C_2)=w(C_1)$, so their contributions to $\mu_t(x_1,x_2)$ cancel. 

    The triple $\langle x_1,x_2,x_3\rangle$ is an element of $MC_{2,2}(X)$ and $d_2(\langle x_1,x_2,x_3\rangle ) = \langle x_1,x_3\rangle $. This last equality means that $d(x_1,x_2)+d(x_2,x_3) = d(x_1,x_3)$, which entails $w(C_1) = w(C_2)$ above. In turn, the signature of each connection depends on the degree of the magnitude homology chain given by the path. 
    \end{example}

    There is an analogous cancellation of terms in \eqref{eq:zeta_metric_with_det}: the linear subdigraphs are related by local modifications that come from transpositions.

    \begin{example}We keep the notations of the preceding example. 
    There is a linear subdigraph $L_1$ given by the circuit $x_1 \to x_2 \to x_3 \to x_6 \to x_5 \to x_4\to x_1$, which has length $6$, hence weight $e^{-6t}$, and signature $(-1)^6+1 = -1$. If one swaps the roles of $x_3$ and $x_5$ in $L_1$, one gets a linear connection $L_2$ with circuits $x_1 \to x_2 \to x_5 \to x_4 \to x_1$ and $x_3 \to x_6 \to x_3$, which has the same total length but opposite signature. (This kind of operation appears in combinatorial linear algebra, specifically in the proof that the transposition of two rows or two columns of a matrix changes its determinant by a factor $-1$, see \cite[Thm. 4.3.2]{Brualdi2008}.)
    \end{example}
    
The example suggests that it might be possible to introduce $(\mathbb N, \mathbb R)$-bigraded homology theories $ H^C_{*,*}$ and $ H^L_{*,*}$ such that 
\begin{equation}\label{eq:quotient_of_alternating_ranks}
    |tX| = \frac{\sum_{\ell\in \Rr} \left(\sum_{n\geq 0} (-1)^n \operatorname{rank} H_{n,\ell}^C \right) q^\ell }{\sum_{\ell\in \Rr} \left(\sum_{n\geq 0} (-1)^n \operatorname{rank} H_{n,\ell}^L \right) q^\ell }, 
\end{equation}
where $q=e^{-t}$. 
For this to hold, $H_{\ast, \ell}^L$ should be the homology of a complex $(MC^L_n,d^L_n)_n$ whose chains are linear subdigraphs of total weight $e^{-\ell}$, the degree of a linear subdigraph $L$ being $n=\#X - c(L)$, and whose boundary operator comes from transpositions that preserve the total weight. In turn,  $H_{\ast, \ell}^C$ should be the homology of a complex $(MC^C_n,d^C_n)$ whose chains are connections of total weight $e^{-\ell}$, the degree of a connection $C$ being $n=\# X -c(L) + 1$, and whose boundary operator comes from an appropriate combination of $d^L_\ast$ and $d_r$, acting respectively on the linear subdigraph of the connection and on the path of the connection ($r$ denotes the length of this path). Formula \eqref{eq:quotient_of_alternating_ranks} would then follow by arguments analogous to those used to establish Theorem 3.5 in  \cite{Leinster2021}. Unlike $(MC_n,d_n)$, the complexes $(MC^L_n,d^L_n)_n$and $(MC^C_n,d^C_n)_n$ would be bounded.

\section*{Acknowledgement}

The author thanks Matilde Marcolli for her support and her valuable comments. 

\bibliographystyle{ieeetr}
\bibliography{biblio}

\begin{thebibliography}{10}

\bibitem{Leinster2008}
T.~Leinster, ``The {Euler} characteristic of a category,'' {\em Documenta
  Mathematica}, vol.~13, pp.~21--49, 2008.

\bibitem{Leinster2013}
T.~Leinster, ``The magnitude of metric spaces,'' {\em Documenta Mathematica},
  vol.~18, pp.~857--905, 2013.

\bibitem{magnitude-bibliography}
T.~Leinster, ``Magnitude: a bibliography.''
  {\url{https://www.maths.ed.ac.uk/~tl/magbib/}}.
\newblock Version: 11 July 2024. Accessed: 11 July 2024.

\bibitem{Meckes2015}
M.~W. Meckes, ``Magnitude, diversity, capacities, and dimensions of metric
  spaces,'' {\em Potential Analysis}, vol.~42, pp.~549--572, 2015.

\bibitem{Barcelo2018}
J.~A. Barcel{\'o} and A.~Carbery, ``On the magnitudes of compact sets in
  {E}uclidean spaces,'' {\em American Journal of Mathematics}, vol.~140, no.~2,
  pp.~449--494, 2018.

\bibitem{Gimperlein2021}
H.~Gimperlein and M.~Goffeng, ``On the magnitude function of domains in
  {E}uclidean space,'' {\em American Journal of Mathematics}, vol.~143, no.~3,
  pp.~939--967, 2021.

\bibitem{Akkaya2023}
M.~Akkaya and {\"O}.~{\"U}nl{\"u}, ``The {Euler} characteristic of finite
  categories,'' {\em arXiv preprint arXiv:2301.08966}, 2023.

\bibitem{Chen2023formula}
S.~Chen and J.~P. Vigneaux, ``A formula for the categorical magnitude in terms
  of the {Moore-Penrose} pseudoinverse,'' {\em Bulletin of the Belgian
  Mathematical Society - Simon Stevin}, vol.~30, pp.~341--353, Nov. 2023.
\newblock To appear in Vol. 30, Issue 3.

\bibitem{Rota1964}
G.-C. Rota, ``On the foundations of combinatorial theory {I}. {Theory} of
  {M}{\"o}bius functions,'' {\em Probability theory and related fields},
  vol.~2, no.~4, pp.~340--368, 1964.

\bibitem{Leinster2021}
T.~Leinster and M.~Shulman, ``Magnitude homology of enriched categories and
  metric spaces,'' {\em Algebraic \& Geometric Topology}, vol.~21, no.~5,
  pp.~2175--2221, 2021.

\bibitem{Brualdi2008}
R.~Brualdi and D.~Cvetkovic, {\em A Combinatorial Approach to Matrix Theory and
  Its Applications}.
\newblock Discrete Mathematics and Its Applications, CRC Press, 2008.

\bibitem{Bourbaki2007algebre}
N.~Bourbaki, {\em Alg{\`e}bre. Chapitres 1 {\`a} 3.}
\newblock {\'E}l{\'e}ments de math{\'e}matique, Springer, 2007.
\newblock Reimpression of the 1970 edition.

\bibitem{Roff2023small}
E.~Roff and M.~Yoshinaga, ``The small-scale limit of magnitude and the
  one-point property,'' {\em arXiv preprint arXiv:2312.14497}, 2023.

\bibitem{Leinster2023spaces}
T.~Leinster and M.~Meckes, ``Spaces of extremal magnitude,'' {\em Proceedings
  of the American Mathematical Society}, vol.~151, no.~09, pp.~3967--3973,
  2023.

\bibitem{Penrose1955}
R.~Penrose, ``A generalized inverse for matrices,'' {\em Mathematical
  Proceedings of the Cambridge Philosophical Society}, vol.~51, no.~3,
  pp.~406--413, 1955.

\bibitem{BenIsrael2003}
A.~Ben-Israel and T.~Greville, {\em Generalized Inverses: Theory and
  Applications}.
\newblock CMS Books in Mathematics, Springer, 2003.

\end{thebibliography}

\end{document}